\documentclass[11pt]{article}

\usepackage{latexsym,color,amsmath,amsthm,amssymb,amscd,amsfonts}
\usepackage{amsthm}
\usepackage{graphicx}
\usepackage{cancel,url}

\newtheorem{conj}{Conjecture}[section]
\newtheorem{theo}{Theorem}[section]

\newtheorem{lem}[conj]{Lemma}
\newtheorem{prop}[conj]{Proposition}
\newtheorem{coro}[conj]{Corollary}

\newcommand\independent{\protect\mathpalette{\protect\independent}{\perp}} 
\def\independent#1#2{\mathrel{\rlap{$#1#2$}\mkern2mu{#1#2}}}

\newcommand{\R}{\mathbb{R}}

\renewcommand{\P}{\mathbb{P}}

\newcommand{\E}{\mathbb{E}}

\DeclareMathOperator{\Var}{Var}

\newcommand{\Z}{\mathcal{Z}}

\newcommand{\eps}{\varepsilon}

\renewcommand{\Z}{\mathbb{Z}}

\newcommand{\N}{\mathbb{N}}

\date{\vspace{-5ex}}

\author{Heshan Aravinda, Arnaud Marsiglietti, James Melbourne}

\date{}

\begin{document}

\title{Concentration Inequalities for Ultra Log-Concave Distributions} 

\maketitle

\begin{abstract}
    
We establish concentration inequalities in the class of ultra log-concave distributions. In particular, we show that ultra log-concave distributions satisfy Poisson concentration bounds. As an application, we derive concentration bounds for the intrinsic volumes of a convex body, which generalizes and improves a result of Lotz, McCoy, Nourdin, Peccati, and Tropp (2019).

\end{abstract}


\vskip5mm
\section{Introduction}

A random variable $X$ taking values in the set of natural numbers $\N = \{0, 1, 2, \dots\}$ is called log-concave if its probability mass function $p$ satisfies
$$ p(n)^2 \geq p(n-1)p(n+1), $$
for all integers $n \in \N$, and $X$ has contiguous support. For example, Bernoulli, binomial, geometric, and Poisson distributions are all log-concave. This class of discrete distributions appears naturally in probability theory and combinatorics (see, e.g., \cite{JP}, \cite{Br}, \cite{St}, \cite{Pe}, \cite{B}), and has been an object of recent study \cite{BMM, MT}. 

Ultra log-concave random variables, whose probability mass function satisfies
$$ p(n)^2 \geq \frac{n+1}{n} p(n-1)p(n+1), \quad n \in \N $$
on a contiguous support, form an important subclass of log-concave random variables. Equivalently, such discrete random variables are log-concave with respect to the Poisson distribution. Examples include sums of independent binomial with arbitrary parameters and Poisson distributions. Ultra log-concave distributions share important features. For example, the sum of independent ultra log-concave distributions remains ultra log-concave (see \cite{L}, \cite{G}), and an information theoretic characterization of the Poisson distribution as maximizing the Shannon entropy among ultra log-concave distributions under a mean constraint is proven in \cite{J} (see also \cite{JKM}). Many sequences have been proved to be ultra log-concave in graph and matroid theory (see, e.g., \cite{Ma}, \cite{H}, \cite{CS}, \cite{Le}, \cite{ALGV}). Recently, a fundamental ultra log-concave sequence in convex geometry, namely the intrinsic volumes of a convex body, has been studied in \cite{LMNPT}, where concentration inequalities are established. Concentration of (conic) intrinsic volumes have been demonstrated to rigorously explain threshold phenomena in recovery of signals and convex optimization problems with random data (see \cite{ALMT}).

In this article, we establish concentration inequalities for all ultra log-concave sequences. Our main result extends the classical Poisson tail bounds to the whole class of ultra log-concave distributions. Let us recall the following function
\begin{equation}\label{Ben}
    h(x) = 2 \frac{(1+x) \log(1+x) - x}{x^2}, \qquad x \in [-1, +\infty),
\end{equation}
sometimes called the Bennett function.

\begin{theo}\label{concentration}

Let $X$ be an ultra log-concave random variable. Then,
$$ \P(X-\E[X] \geq t) \leq e^{-\frac{t^2}{2 \E[X]} h(\frac{t}{\E[X]})}, \qquad \forall \, t \geq 0, $$
and
$$ \P(X - \E[X] \leq - t) \leq e^{-\frac{t^2}{2 \E[X]} h(-\frac{t}{\E[X]})}, \qquad \forall \, 0 \leq t \leq \E[X]. $$

\end{theo}

In the special case that $X$ is Poisson, Theorem \ref{concentration} is folklore. For this special case, an alternate proof can be obtained from an application of Bennett's inequality to i.i.d. sums of Bernoulli random variables as outlined in \cite{Pol} as exercise 16 (see also \cite{BLM}, \cite{C}, \cite{V}).

Using the standard properties of the Bennett function, namely that $h$ decreases from $2$ to $0$, $h(0) = 1$, and $h(x) \geq 1/(1+x)$ for $x \geq 0$, and the trivial fact that $\P(X - \E[X] \leq - t) = 0$ when $t > \E[X]$ since $X$ is supported on $\N$, we deduce a sub-Gaussian bound for the small deviations of an ultra log-concave random variable, and a sub-exponential bound for the large deviations.

\begin{coro}\label{bound}

Let $X$ be an ultra log-concave random variable. Then, for all $t \geq 0$,
$$ \P(X-\E[X] \geq t) \leq e^{-\frac{t^2}{2(t + \E[X])}}, $$
and
$$ \P(X - \E[X] \leq - t) \leq e^{-\frac{t^2}{2 \E[X]}}. $$

\end{coro}

Since the sum of independent ultra log-concave random variables is ultra log-concave, Corollary \ref{bound} applies to $X = S_n$, where $S_n = \sum_{k=1}^n X_k$, $n \geq 1$, with $X_k$'s independent ultra log-concave.

Let us note that Bennett type bounds have been established  under a discrete Bakry-\'Emery type log-concavity, called $c$-log-concavity, by Johnson \cite{Jo}. Proposition 8.1 in \cite{Jo} gives the bound
\begin{equation}\label{J}
\P(X \geq \E[X] + t) \leq e^{-\frac{ct^2}{2} h(ct)},
\end{equation}
for all $c$-log-concave random variables. According to \cite[Lemma 5.1]{Jo} and \cite[Lemma 5.3]{Jo}, if $X$ is ultra log-concave then $X$ is $c$-log-concave and $c \leq \frac{1}{\E[X]}$. Since the function $x \mapsto x h(x)$ is increasing, we deduce that
$$ e^{-\frac{t^2}{2 \E[X]} h(\frac{t}{\E[X]})} \leq e^{-\frac{ct^2}{2} h(ct)}. $$
Therefore, our bounds in Theorem \ref{concentration} are always stronger than the bound \eqref{J} from \cite{Jo} for all ultra-log-concave distributions. In fact, for any $\eps>0$, there is an ultra log-concave distribution $Z$ such that $Z$ is $c$-log-concave with $c < \eps$ and $\frac{1}{\E[Z]}$ is bounded away from 0 (take, for example, a truncated Poisson distribution with parameter tending to $+\infty$). For such ultra log-concave distributions, the bound \eqref{J} becomes weak.

Our results imply concentration bounds for the intrinsic volumes of a convex body in terms of their central intrinsic volume without reference to the ambient dimension (see Section 3 for the details). In the special case that $Z_K$ is the intrinsic volume random variable associated to $K$ convex body in $\mathbb{R}^n$, an immediate application of Corollary \ref{bound} yields the bound
$$ \P(|Z_K - \E[Z_K]| \geq t \sqrt{n}) \leq 2e^{-\frac{t^2}{2}}, \qquad 0 \leq t \leq \sqrt{n}, $$
see Corollary \ref{intrinsic}. This improves the bounds $2e^{\frac{-3t^2}{28}}$ from Lotz, McCoy, Nourdin, Peccati, and Tropp \cite{LMNPT}.

As we will see in Section 3, our bounds give considerable improvement when the central intrinsic volume is of order less than $n$, since if $\E[Z_K]$ is less than, say, $\sqrt{n}$, then our bounds improve as the dimension grows to infinity, hence reflecting the high-dimensional aspect of the concentration bound. Specifically, if $\E[Z_K] \leq \sqrt{n}$, Corollary \ref{bound} yields
$$ \P(|Z_K - \E[Z_K]| \geq t \sqrt{n}) \leq 2e^{-\frac{t^2}{2(t+1)}\sqrt{n}}, $$
which goes to $0$ exponentially fast with the dimension, for any fixed $t > 0$.

In high dimension, Poisson type concentration bounds are optimal for intrinsic volume random variables. This can be seen by considering the following scaling of the unit cube $K = \frac{\lambda}{n-\lambda} [0,1]^n$ for $\lambda >0$.  In this case, the intrinsic volume random variable $Z_K(n)$ is Binomial with parameters $n$ and $p = \frac \lambda n$, so that $Z_K(n)$ tends to a Poisson with parameter $\lambda$ with $n \to \infty$.

Our main theorem also implies the following estimate on the variance of an arbitrary ultra log-concave random variable.

\begin{prop}\label{var}

Let $X$ be an ultra log-concave random variable. Then,
$$ \Var(X) \leq \E[X], $$
with equality when $X$ is Poisson.

\end{prop}

In short, among ultra log-concave variables of fixed expectation, the Poisson has maximum variance. In the context of intrinsic volumes, Proposition \ref{var} says that the variance of an intrinsic volume random variable is less than its central intrinsic volume.  Moreover, this variance bound is independent of the dimension of the ambient Euclidean space that a convex body is embedded within.

Bounded ultra log-concave random variables, say $X$ supported on $\{0, \dots, n\}$, satisfy $\Var(X) \leq n$, by application of the trivial inequality $\mathbb{E}[X] \leq n$. Applying to the particular case that $X$ is the intrinsic volume random variable associated to convex body in $\mathbb{R}^n$, improves the bound $\Var(Z_K) \leq 4n$ from \cite{LMNPT}.  Again, since an arbitrary Binomial can be realized as the distribution of an intrinsic volume random variables, Proposition \ref{var} is sharp for intrinsic volume random variables.

The article is organised as follows. In Section 2 we prove our main result Theorem \ref{concentration}, and derive Proposition \ref{var}. In Section 3 we apply our results to study the intrinsic volume random variables.

\section{Proofs of Theorem \ref{concentration} and Proposition \ref{var}}

The main ingredient of proofs is the following bound on the moment generating function of an ultra log-concave distribution.

\begin{lem}\label{main}

Let $X$ be an ultra log-concave random variable. Then, for all $t \in \R$,
\begin{equation}\label{MGF}
    \E[e^{tX}] \leq \E[e^{tZ}],
\end{equation} 
where $Z$ is a Poisson distribution with parameter $\E[X]$.

\end{lem}

Lemma \ref{main} tells us that Poisson distributions maximize the moment generating function in the class of ultra-log-concave random variable, under a mean constraint.

The proof of Lemma \ref{main} rely on a discrete localization technique recently developed by the second and third named authors in \cite{MM} (see also \cite{KLS}, \cite{FG04}, \cite{FG06}, \cite{LS}, \cite{E}, \cite{BM}, \cite{BM2}, \cite{K} for other aspects of the localization technique). The main idea is to reduce the desired inequality to extreme points, and then prove the inequality for these extreme points. Let us recall the results from \cite{MM}.

For $M,N \in \Z$, denote $[M,N] = \{M, \dots, N\}$. Let $M, N \in \Z$. Let us denote by $\mathcal{P}([M,N])$ the set of all probability measures supported on $[M,N]$. Let $\gamma$ be a measure with contiguous support on $\mathbb{Z}$, and let $h \colon [M,N] \to \mathbb{R}$ be an arbitrary function. Let us consider $\mathcal{P}_h^{\gamma}([M,N])$ the set of all distributions $\P_X$ in $\mathcal{P}([M,N])$, log-concave with respect to $\gamma$, and satisfying $\E[h(X)] \geq 0$, that is,
$$ \mathcal{P}_h^{\gamma}([M,N]) = \{ \P_X \in \mathcal{P}([M,N]) : X \mbox{ log-concave with respect to } \gamma, \, \E[h(X)] \geq 0 \}. $$

Consider probability mass functions of the form

    \begin{equation}\label{extremizers}
        p(n) = C p^n q(n)1_{[k, l]}(n) ,
    \end{equation}
    for some $C, p >0$, $k,l \in [M,N]$, where $q$ is the mass function of $\gamma$.

\begin{theo}[\cite{MM}]\label{maximize}

Let $\Phi \colon \mathcal{P}_h^{\gamma}([M,N]) \to \R$ be a convex function. Then
$$ \sup_{\P_X \in \mathcal{P}_h^{\gamma}([M,N])} \Phi(\P_X) \leq \sup_{\P_{X} \in \mathcal{A}_h^{\gamma}([M,N])} \Phi(\P_{X}), $$
where $\mathcal{A}_h^{\gamma}([M,N]) =  \mathcal{P}_h^{\gamma}([M,N]) \cap \{ \P_{X} : X \mbox{ with probability mass function of the form \eqref{extremizers}} \}$.

\end{theo}

\begin{proof}[Proof of Lemma \ref{main}]
Fix an ultra log-concave random variable $X_0$. By approximation, one may assume that $X_0$ is compactly supported, say on $\{M, \dots, N\}$. Fix $t \in \R$. By Theorem \ref{maximize}, applied to $\Phi(\P_X) = \E[e^{tX}]$, which is linear, and to the constraint function $h(n) = \E[X_0] - n$, it suffices to prove inequality \eqref{MGF} for ultra log-affine random variables with respect to the Poisson measure, that is, for distributions of the form \eqref{extremizers} with $q(n)=1/n!$:
\begin{align*}
    p(n) = C \frac{p^n}{n!} 1_{[k, l]}(n)
\end{align*}
for $p > 0$ and $0 \leq k \leq l$, where $\frac 1 C =  {\sum_{n=k}^l \frac{p^n}{n!}}$. In this case we wish to prove for $t \in \mathbb{R}$,
\begin{align*}
    C \sum_{n=k}^l \frac{e^{tn}p^n}{n!} = \mathbb{E}[e^{tX}] \leq e^{\mathbb{E}[X] (e^t - 1)} = e^{C \sum_{n=k}^l \frac{n p^n}{n!} (e^t-1)}.
\end{align*}
Write $y = e^{t}$ and define, for $K,L \in \mathbb{Z}$,
$$ \Psi_{K,L}(x) = \sum_{n=K}^{L} \frac{x^n}{n!} $$
if $L \geq K \geq 0$, $\Psi_{K,L} = \Psi_{0, L}$ if $K \leq 0 \leq L$, and $\Psi_{K,L} = 0$ if $L < 0$. Note that
$$ \sum_{n=k}^l \frac{n p^n}{n!} = p \sum_{n=k-1}^{l-1} \frac{p^n}{n!} = p \Psi_{k-1,l-1}(p), \quad C = \frac{1}{\Psi_{k,l}(p)}, \quad \sum_{n=k}^l \frac{e^{tn} p^n}{n!} = \sum_{n=k}^l \frac{(yp)^n}{n!} = \Psi_{k,l}(yp). $$
Thus we wish to prove that for all $y \geq 0$,
\begin{align*}
    e^{\frac{p\Psi_{k-1,l-1}(p)}{\Psi_{k,l}(p)}(y-1)} \geq \frac{\Psi_{k,l}(yp)}{\Psi_{k,l}(p)}.
\end{align*}
Taking logarithms and rearranging, this is equivalent to
\begin{align*}
    f(y) := \frac{p\Psi_{k-1,l-1}(p)}{\Psi_{k,l}(p)}(y-1) - \log \Psi_{k,l}(yp) + \log {\Psi_{k,l}(p)} \geq 0.
\end{align*}
To this end, we claim that $f(1) = f'(1) = 0$ and that $f(y)$ is convex. Observe that
$$ f(1) = 0 - \log {\Psi_{k,l}(p)} + \log {\Psi_{k,l}(p)} = 0, $$ and
$$ \frac{d}{dx} \Psi_{k,l}(x) = \sum_{n=k}^l \frac{n x^{n-1}}{n!} = \sum_{n=k-1}^{l-1}\frac{x^n}{n!} = \Psi_{k-1,l-1}(x). $$  Thus,
\begin{align*}
    f'(y) = \frac{p\Psi_{k-1,l-1}(p)}{\Psi_{k,l}(p)} - \frac{p \Psi_{k-1,l-1}(yp)}{\Psi_{k,l}(yp)}
\end{align*}
and $f'(1) = 0$. Finally,
\begin{align*}
    f''(y) =  - p^2 \frac{\Psi_{k,l}(yp) \Psi_{k-2,l-2}(yp) - \Psi_{k-1,l-1}^2(yp)}{\Psi_{k,l}^2(yp)}.
\end{align*}
Note that for fixed $x \geq 0$ and $m \in \N$, $z_l = \sum_{n=l-m}^l \frac{x^n}{n!}$, $l \geq m$, can be seen as the convolution of the sequence $\{x_n\}$ and $\{y_n\}$, where $x_n = \frac{x^n}{n!}$ for $n \geq 0$ and $x_n = 0$ for $n<0$, and $y_n = 1$ for $n \in \{0, \dots, m\}$ and $y_n = 0$ otherwise. Both $\{x_n\}$ and $\{y_n\}$ are log-concave sequences. Therefore, their convolution is also log-concave \cite{Ho}. Hence, $z_l$ is log-concave in $l$, and by choosing $m = l-k$ in the inequality $z_{l-1}^2 \geq z_l z_{l-2}$, we deduce that
$$ \Psi_{k,l}(yp) \Psi_{k-2,l-2}(yp) - \Psi_{k-1,l-1}^2(yp) \leq 0. $$
Therefore, we conclude that $f''(y) \geq 0$, and our result follows.
\end{proof}

We now derive Theorem \ref{concentration}.

\begin{proof}[Proof of Theorem \ref{concentration}]
Let $X$ be an ultra log-concave random variable. It is standard that concentration bounds follow from upper bounds on the moment generating function (see, e.g., \cite{BLM}, \cite{V}). Let $x \geq 0$. Applying Markov's inequality, we have for all $t>0$
$$ \P(X \geq \E[X] + x) = \P(e^{tX} \geq e^{t(\E[X] + x)}) \leq e^{-t(\E[X] + x)} \E[e^{tX}] \leq e^{-t(\E[X] + x)} e^{\mathbb{E}[X] (e^t - 1)}, $$
where the last inequality comes from Lemma \ref{main}. Optimizing over $t>0$, that is, taking $t = \log(1+\frac{x}{\E[X]})$, yields
$$ \P(X \geq \E[X] + x) \leq e^{-\frac{x^2}{2 \E[X]} h(\frac{x}{\E[X]})}, $$
where the function $h$ is defined in \eqref{Ben}. The argument is similar for the small deviations. Let $x \in [0, \E[X]]$. Then, for all $t>0$,
\begin{equation}\label{mid}
    \P(X \leq \E[X] - x) = \P(e^{-tX} \geq e^{-t(\E[X]-x)}) \leq e^{t(\E[X]-x)} e^{\E[X](e^{-t} - 1)}.
\end{equation} 
Optimizing over all $t>0$, that is, taking $t=-\log(1-\frac{x}{\E[X]})$ when $x < \mathbb{E}[X]$, yields
\begin{equation*}
    \P(X \leq \E[X] - x) \leq e^{-\frac{x^2}{2 \E[X]} h(-\frac{x}{\E[X]})}.
\end{equation*} 
When $x = \mathbb{E}[X]$, taking $t \to \infty$ in \eqref{mid} yields the result, as $h(-1) = 2$.
\end{proof}

We also deduce Proposition \ref{var} from Lemma \ref{main}.

\begin{proof}[Proof of Proposition \ref{var}]
    If $X$ is ultra log-concave and $\mathbb{E}X = \lambda > 0$, then by Lemma \ref{main}
    \begin{align*}
        \mathbb{E} e^{tX} \leq \mathbb{E} e^{t Z}
    \end{align*}
    where $Z$ is a Poisson random variable with parameter $\lambda$.  Expanding the inequality, this is
    \begin{align*}
        1 + \lambda t + \frac{t^2}{2} \E[X^2]  \leq 1 + \lambda t + \frac{t^2}{2} \E[Z^2] + o(t^2).
    \end{align*}
    Cancelling terms, dividing by $t^2$ and taking $t \to 0$ gives
    $
        \mathbb{E}[X^2] \leq \mathbb{E}[Z^2]
    $
    so that 
    \begin{align*}
        \Var(X) \leq \Var(Z) = \lambda.
    \end{align*}
\end{proof}

\section{Concentration for the intrinsic volumes of a convex body}

In this section, we apply our results to obtain concentration bounds for the intrinsic volumes of a convex body $K$. Our bounds improve upon \cite{LMNPT}.

For two convex bodies $K, E \subset \mathbb{R}^n,\,$ and $t \in \mathbb{R}^{+},\,$  the mixed volumes of $K$ and $E$, denoted by $V_i(K,E)$, are defined as the coefficients of the following polynomial describing the volume ($n$-dimensional Lebesgue measure) of Minkowski sum of $K$ and $tE,\,$
\begin{equation}\label{steiner}
V(K + tE) = \sum_{i=0}^n\, \binom{n}{i}\,V_i\, (K,E)\,t^i.
\end{equation}
Let $B_2^n$ denote the $n$-dimensional Euclidean unit ball and $\kappa_n$ denote the volume of $B_2^n$. When $E=B_2^n$, the polynomial \eqref{steiner} becomes the Steiner polynomial which can be written via the normalization $V_i(K) = \binom{n}{i}\,V_{n-i} (K,B_2^n)/ \kappa_{n-i}$ as,
$$ V(K + tB_2^n) = \sum_{i=0}^n\,\kappa_{n-i} V_i(K)\,t^{n-i}, $$ 
Here, $V_i(K)$ is called the $i$-th intrinsic volume of $K$. We refer to \cite{Sch} for further details about intrinsic volumes.

The total intrinsic volume of a convex body $K$, called the Wills functional, is the quantity
$$ W(K) = \sum_{i=0}^n\, V_i(K). $$
The normalized intrinsic volumes consists in the sequence $\{\widetilde{V}_i(K) : i = 0, \dots, n\}$, where
$$ \widetilde{V}_i(K) = \frac{V_i(K)}{W(K)}. $$
Using the terminology from \cite{LMNPT}, the intrinsic volume random variable $Z_K$ associated with a convex body $K$ in $\R^n$, takes non-negative integer values according to the following distribution,
$$ P(Z_K = j) = \widetilde{V}_j(K), \qquad j=0, \dots,n. $$

It is a consequence of the Alexandrov-Fenchel inequality that the intrinsic volumes of a convex body form an ultra log-concave sequence \cite{Mc} (see also \cite{L}, \cite{G}). Consequently, the intrinsic volume random variable $Z_K$ is ultra log-concave. Therefore, Corollary \ref{bound} tells us the following.

\begin{coro}\label{intrinsic}

Let $K$ be a convex body in $\R^n$, and let $Z_K$ be the intrinsic volume random variable associated with $K$. Then, for all $0 \leq t \leq \sqrt{n}$,
$$ \P(|Z_K - \E[Z_K]| \geq t \sqrt{n}) \leq 2e^{-\frac{t^2}{2}}. $$

\end{coro}

\begin{proof}
If $n < t\sqrt{n}+\E[Z_K]$, then one trivially has $\P(Z_K - \E[Z_K] \geq t \sqrt{n}) = 0$ since $Z_K \leq n$. If $n \geq t\sqrt{n}+\E[Z_K]$, then by Corollary \ref{bound},
$$ \P(Z_K - \E[Z_K] \geq t \sqrt{n}) \leq e^{-\frac{t^2}{2} \frac{n}{t\sqrt{n} + \E[Z_K]}} \leq e^{-\frac{t^2}{2}}. $$
For the small deviations, since $\E[Z_k] \leq n$, by Corollary \ref{bound},
$$ \P(Z_K - \E[Z_K] \leq -t \sqrt{n}) \leq e^{-\frac{t^2}{2} \frac{n}{\E[Z_K]}} \leq e^{-\frac{t^2}{2}}. $$
\end{proof}

On the other hand, if $\E[Z_K]$ is of order $\sqrt{n}$, say $\E[Z_K] \leq c\sqrt{n}$ for some absolute constant $c>0$, Corollary \ref{intrinsic} readily yields the bound
$$ \P(|Z_K - \E[Z_K]| \geq t \sqrt{n}) \leq 2e^{-\frac{t^2\sqrt{n}}{2(t+c)}}. $$
In fact, any $o(n)$ bound for $\E[Z_K]$ will yield a concentration bound that decay to 0 as the dimension grows to $\infty$, for any fixed $t>0$.

Note that if $K$ is a fixed convex body in $\R^n$, there is always a dilation parameter $r > 0$ (possibly dependent on the dimension) such that $\E[Z_{rK}] = o(n)$. As noted in \cite{LMNPT}, the intrinsic volumes random variables of the scaled unit cube $r[0,1]^n$ has a binomial distribution with parameters $n$ and $\frac{r}{1+r}$, therefore
$$ \E[Z_{r[0,1]^n}] = n\frac{r}{1+r}. $$
Here, taking $r=o(1)$ will yield $\E[Z_{r[0,1]^n}] = o(n)$.

\vskip2cm


\noindent Heshan Aravinda \\
Department of Mathematics \\
University of Florida \\
Gainesville, FL 32611, USA \\
heshanaravinda.p@ufl.edu

\vspace{0.8cm}

\noindent Arnaud Marsiglietti \\
Department of Mathematics \\
University of Florida \\
Gainesville, FL 32611, USA \\
a.marsiglietti@ufl.edu

\vspace{0.8cm}

\noindent James Melbourne \\
Probabilidad y Estad\'istica \\
Centro de Investigaci\'on en Matem\'aticas (CIMAT) \\
Guanajuato, Gto 36023, M\'exico\\
james.melbourne@cimat.mx

\end{document}